\documentclass[12pt]{amsart}
\usepackage[latin9]{inputenc}
\usepackage{xcolor}
\usepackage{amstext}
\usepackage{amsthm}
\usepackage{amssymb}
\usepackage[unicode=true,
 bookmarks=false,
 breaklinks=false,pdfborder={0 0 1},backref=section,colorlinks=false]{hyperref}
\usepackage{xcolor}
\usepackage{verbatim}
\usepackage{amstext}
\usepackage{amsthm}
\usepackage{eurosym}
\usepackage{latexsym}
\usepackage{amsthm}
\usepackage{amsfonts}
\usepackage{ulem}
\allowdisplaybreaks

\newtheorem{theorem}{Theorem}

\newtheorem{lemma}[theorem]{Lemma}
\newtheorem*{lemma*}{Lemma}

\newtheorem{proposition}[theorem]{Proposition}
\newtheorem*{proposition*}{Proposition}

\theoremstyle{definition}
\newtheorem{definition}[theorem]{Definition}

\newtheorem*{acknowledgements*}{Acknowledgements}
\theoremstyle{remark}

\newcommand{\N}{\mathbb N}

\newcommand{\Z}{\mathbb Z}

\title
[Symmetric finite representability  of r.i.\ spaces] 
{Symmetric finite representability of $\ell^p$-spaces 
in rearrangement invariant spaces on $[0,1]$} 

\author[S.~V. Astashkin]{Sergey V. Astashkin}
\address{Department of Mathematics, Samara National Research University, 
Moskovskoye shosse 34, 443086, Samara, Russia}
\email{astash56@mail.ru}

\author[G.~P. Curbera]{Guillermo P. Curbera}
\address{Facultad de Matem\'aticas \& IMUS,
Universidad de Sevilla, 
Calle Tarfia s/n,  Sevilla 41012, Spain}
\email{curbera@us.es}

\thanks{The work of the first author was completed as a part of the implementation of the development program of the Scientific and Educational Mathematical Center Volga Federal District, agreement
no. 075-02-2022-878. The second author  acknowledges the support  of 
PGC2018-096504-B-C31 (Spain).}
\date{\today}

\begin{document}

\begin{abstract}
For a separable rearrangement invariant space $X$ on $[0,1]$ of fundamental 
type we identify the set of all $p\in [1,\infty]$ such that $\ell^p$ is finitely represented 
in $X$ in such a way that the unit basis vectors of $\ell^p$ ($c_0$ if $p=\infty$) 
correspond to pairwise disjoint and equimeasurable functions.  
This can be treated as a follow up of a paper by the first-named 
author related to separable rearrangement invariant spaces on $(0,\infty)$.
\end{abstract} 

\maketitle

{\let\thefootnote\relax\footnotetext{File: \jobname.pdf}}


\section{Introduction}
\label{S1}


Given  a Banach space $X$, recall that  $\ell^p$, for $1\le p\le\infty$,  is said to be
\textit{finitely represented} in $X$ if for every $n\in\mathbb{N}$ and $\varepsilon>0$ 
there exist $x_1,x_2,\dots,x_n\in X$ such that for every 
$a=(a_k)_{k=1}^n\in\mathbb{R}^n$ we have 
\begin{equation*}
(1+\varepsilon)^{-1}\|a\|_p\le \Big\|\sum_{k=1}^na_kx_k\Big\|_X
\le (1+\varepsilon)\|a\|_p,
\end{equation*}
where $\|a\|_p:=\big(\sum_{k=1}^n |a_k|^p\big)^{1/p}$, for $1\le p<\infty$, and
$\|a\|_\infty:=\sup_{1\le k\le n}|a_k|$, for $p=\infty$.
A  celebrated result by Dvoretzky showed that  
$\ell^2$ is finitely represented in an arbitrary infinite-dimensional 
Banach space $X$; see \cite{Dvor} or \cite[Theorem 11.3.13]{AK}.

A related, though more restrictive concept,  is the block finite  representability. Given a Banach space $X$ and $\{z_i\}_{i=1}^\infty$, a bounded sequence in $X$,
the space $\ell^p$, for $1\le p\le\infty$, is said to be 
\textit{block finitely represented in $\{z_i\}_{i=1}^\infty$}
if for every $n\in\mathbb{N}$ and $\varepsilon>0$ there exist 
$0=m_0<m_1<\dots <m_n$ and $\alpha_i\in\mathbb{R}$ such that the vectors $u_k=\sum_{i=m_{k-1}+1}^{m_k}\alpha_iz_i$, $k=1,2,\dots,n$, satisfy the inequality
\begin{equation}\label{eq-1}
(1+\varepsilon)^{-1}\|a\|_p\le \Big\|\sum_{k=1}^n
a_ku_k\Big\|_X\le (1+\varepsilon)\|a\|_p
\end{equation}
for arbitrary $a=(a_k)_{k=1}^n\in\mathbb{R}^n$. Krivine proved, for 
an arbitrary normalized sequence $\{z_i\}_{i=1}^\infty$ in a Banach space $X$ 
such that $\{z_i:i\ge1\}$ is not a relatively compact set, that $\ell^p$ is 
block finitely represented in $\{z_i\}_{i=1}^\infty$ for some $p$ with $1\leq p\leq \infty$;
see, for example, \cite{Kriv}, \cite{Ros} and \cite[Theorem 11.3.9]{AK}.

A next important step has been made by Maurey and Pisier showing in their seminal paper \cite{MaPi} that, for every infinite-dimensional Banach space $X$, the spaces $\ell_{p_X}$ and $\ell_{q_X}$ are finitely represented in $X$, where $p_X:=\sup\{p\in [1,2]\,:\,X\enskip\text{has type}\enskip p\}$ and $q_X:=\inf\{q\in [2,\infty]\,:\,X\enskip\text{has cotype}\enskip q\}$. Later on, in \cite{Shepp}, Shepp discovered similar connections between the finite representability of $\ell^p$-spaces and the upper and lower estimate notions in the case of Banach lattices (see also \cite{A-11} and references therein).   

In the special case of rearrangement invariant spaces, the study of the finite representability of $\ell^p$-spaces was undertaken in  \cite{A-11}. Recall that measurable functions $x(t)$ and $y(t)$ on the measure space $(I,m)$, where $I=[0,1]$ or $(0,\infty)$ and $m$ is the Lebesgue measure, are \textit{equimeasurable} if it holds:
\begin{equation*}
m\{s\in I:\,|x(s)|>\tau\}=m\{s\in I:\,|y(s)|>\tau\}\;\;\mbox{for all}\;\;\tau>0.
\end{equation*}

Let $X$ be  a rearrangement invariant space on $I$ (for all undefined notions see Section \ref{prel} below). We say that  
$\ell^p$, for $1\le p\le\infty$, is \textit{symmetrically finitely represented} 
in $X$ if for every $n\in\N$ and each $\varepsilon>0$ there exist 
equimeasurable functions $x_k\in X$, $k=1,2,\dots,n$, 
such that $\textrm{supp}\,x_i\cap \textrm{supp}\,x_j=\varnothing$, $i\ne j$, and for any $a=(a_k)_{k=1}^n$ 
\begin{equation}
\label{main1}
(1+\varepsilon)^{-1}\|a\|_p\le \Big\|
\sum_{k=1}^na_kx_k\Big\|_X\le (1+\varepsilon) \|a\|_p.
\end{equation}
%
Sometimes the following weaker notion is also considered (see e.g. \cite[p.~264]{AK}). The space $\ell^p$ is \textit{crudely symmetrically finitely represented} 
in a rearrangement invariant space $X$ on $I$ if there exists a constant $C>0$ such that for every $n\in\N$ we can find 
equimeasurable functions $x_k\in X$, $k=1,2,\dots,n$, such that 
$\textrm{supp}\,x_i\cap \textrm{ supp}\,x_j=\varnothing$, $i\ne j$, 
and for every $a=(a_k)_{k=1}^n$ 
\begin{equation*}
C^{-1}\|a\|_p\le \Big\|
\sum_{k=1}^na_kx_k\Big\|_X\le C\|a\|_p.
\end{equation*}
The set of all $p\in [1,\infty]$ such that  $\ell^p$ is symmetrically finitely 
represented (resp.\  crudely symmetrically finitely represented) in $X$ we will denote by ${\mathcal F}(X)$ (resp. ${\mathcal F}_c(X)$). 

If $\alpha_X$ and $\beta_X$ are the Boyd indices of a rearrangement invariant space $X$, it can be easily shown that ${\mathcal F}(X)\subset [1/\beta_X,1/\alpha_X]$ (see also the very beginning of the proof of Theorem \ref{t-2}). In the converse direction, in the book \cite{LT2} (see Theorem 2.b.6) it was stated without proof that $\max {\mathcal F}(X)=1/\alpha_X$ and $\min{\mathcal F}(X)=1/\beta_X$ for every rearrangement invariant space $X$\footnote{The authors intended to present the proof in the next volume of the monograph, which, however, was never published.}.
Later on, a full proof of the mentioned result was given in the paper  \cite{A-11}. More unexpected was the fact that it is possible to give even for some classes of r.i.\ spaces a rather simple {\it complete} description of the set ${\mathcal F}(X)$. In \cite{A-11}, it has been  established for Lorentz spaces on $(0,\infty)$. More recently, in \cite{A-22p}, the latter result has been extended to a wide class of separable rearrangement invariant spaces on $(0,\infty)$ of fundamental type. In particular,  the coincidence of the sets ${\mathcal F}(X)$ and ${\mathcal F}_c(X)$ has been shown. More precisely, the main result of \cite{A-22p}, identifying the latter sets in terms of the Boyd indices and some other  appropriate dilation indices of $X$, can be stated in the following way. 

\begin{theorem}\label{t-1}
For every separable rearrangement invariant space $X$ on $(0,\infty)$ of fundamental type we have:
\begin{itemize}
\item[(i)] If ${\alpha_X^\infty}\le {{\beta_X^0}}$, 
then ${\mathcal F}(X)={\mathcal F}_c(X)=[1/\beta_X,1/\alpha_X]$.
\item[(ii)] If ${\alpha_X^\infty}>{{\beta_X^0}}$, then 
${\mathcal F}(X)={\mathcal F}_c(X)=[1/\beta_X,1/\alpha_X^\infty]
\cup [1/\beta_X^0,1/\alpha_X]$.
\end{itemize}
\end{theorem}
It is worth to note that the condition of being a space of fundamental type is not too restrictive as it is satisfied by most of the 
well-known and important rearrangement invariant spaces (in particular, Orlicz, Lorentz, Marcinkiewicz spaces).

Observe that the method of the proof of Theorem \ref{t-1} is based on using the so-called spreading sequence spaces (see e.g. \cite[Chapter~11]{AK}) and so seems to be not applicable in the case of function spaces on $[0,1]$. 

The main aim of this paper is to present similar description of the set ${\mathcal F}(X)$ for rearrangement invariant spaces on $[0,1]$. Namely, by applying Theorem \ref{t-1} to a suitable extension of a rearrangement invariant space on $[0,1]$ to the semi-axis, we prove the following result.

\begin{theorem}\label{t-2}
Let $X$ be a separable rearrangement invariant space on $[0,1]$ of fundamental type. Then, 
$$
\mathcal{F}(X)=\mathcal{F}_c(X)=[1/{\beta}_X,1/{\alpha}_X].
$$
\end{theorem}

As a consequence of Theorem \ref{t-2}, we obtain a complete  description of the set of all $p\in [1,\infty]$ such that $\ell^p$ is symmetrically finitely represented in a separable Orlicz space and a Lorentz space (see Theorems \ref{Theorem 4b} and \ref{Theorem 4a}).

Along the way, we compliment and refine some constructions related to the definition of partial dilation indices of rearrangement invariant spaces 
on $(0,\infty)$ introduced in \cite{A-22p} (see Proposition \ref{dilation}).

Let us note that the structure of the set ${\mathcal F}(X)$ is closely connected with the spectral theory of operators and so it plays an important role when studying the normal solvability and invertibility of operators between function spaces, as well, in the theory of functional-differential equations and the theory of dynamical systems, etc. (see e.g. \cite{Ant} and references therein). Moreover, the properties of the set ${\mathcal F}(X)$ are used in the study of geometric properties of rearrangement invariant spaces 
(see, e.g., \cite{A16}, \cite{AMT-13}).


\section{Preliminaries}
\label{prel}

\subsection{Banach function and sequence lattices.}
\label{prel1}

We will use the standard definitions and results from the theory of Banach function lattices over a $\sigma$-finite measure space (see \cite{BSh,KA,LT2}).

Let $(\Omega,\Sigma,\mu)$ be a $\sigma$-finite measure space and let $L^0:=L^0(\Omega,\Sigma,\mu)$ be the linear topological space of all  (equivalence classes of) a.e.\ finite real-valued functions defined on $\Omega$ with the natural algebraic operations and the topology of convergence in measure $\mu$ on sets of finite measure. We say that a Banach space  $E\subseteq L^0$ is a {\it Banach lattice} on $\Omega$ if $E$  satisfies the ideal property, that is, we have $y\in E$ and $\|y\|_E\le\|x\|_E$ whenever $x\in E$, $y\in L^0$ and $|y|\le|x|$ a.e.  

If $E$ is a Banach function lattice, then the {\it K\"{o}the dual} 
(or {\it associated}) function lattice $E'$ consists of all $y\in L^0(\Omega,\Sigma,\mu)$ such that
$$
\|y\|_{E'}:=\sup\,\Bigl\{\int_{\Omega}{x(t)y(t)\,d{\mu}}:\;\;
\|x\|_{E}\,\leq{1}\Bigr\}<\infty
$$
(in the case when $E$ is a Banach sequence lattice modelled on $\mathbb{Z}$ the integral should be replaced with the sum over $\mathbb{Z}$).

One can easily check that $E'$ is complete with respect to the norm $y\mapsto \|y\|_{E'}$ and $E$ is continuously embedded into its second 
K\"{o}the dual $E''$, with $\|x\|_{E''}\le\|x\|_E$ for $x\in E$. 
A Banach lattice $E$ {\it has the Fatou property} (or is {\it maximal}) if from $x_n\in E,$ $n=1,2,\dots,$ $\sup_{n=1,2,\dots}\|x_n\|_E<\infty$, $x\in L^0(\Omega,\Sigma,\mu)$ and $x_n\to{x}$ a.e. on $\Omega$ it follows that $x\in E$ and $||x||_E\le \liminf_{n\to\infty}{||x_n||_E}.$ Note that a Banach lattice $E$ has the Fatou property if and only if the natural inclusion of $E$ into $E''$ is a surjective isometry \cite[Theorem~6.1.7]{KA}. 

A Banach lattice $E$ is said to have an {\it order continuous norm} if for every $x\in E$ and any decreasing sequence of sets $A_{n} \in \Sigma$ with $\mu(\bigcap_{n=1}^{\infty} A_n) = 0$ it follows $\|x \chi_{A_{n}} \|_E \rightarrow 0$ as $n \rightarrow \infty$. 

Any K\"{o}the dual lattice $E'$ is embedded isometrically into (Banach) dual space $E^*$ and $E'=E^*$ if and only if $E$ has an order continuous norm \cite[Corollary~6.1.2]{KA}.

\subsection{Rearrangement invariant function spaces}
\label{prel2}

A {\it rearrangement invariant} (in brief, r.i.) (or {\it symmetric}) space $X$ on the measure space $(I,m)$, where $I=[0,1]$ or $(0,\infty)$ and $m$ is the Lebesgue measure, is a Banach function lattice on $I$ satisfying the following condition: if $f\in X$, $g\in L^0(I,m)$ and $g^*\le f^*$,  
then $g\in X$ and $\|g\|_X\le\|f\|_X$. Here and below, $f^*(t)$ is the right-continuous nonincreasing {\it rearrangement} of $|f(s)|$, i.e., 
$$
f^{*}(t):=\inf \{ \tau\ge 0:\,m\{s\in I:\,|f(s)|>\tau\}\le t \},\;\;0<t<m(I).$$
Functions $f^*$ and $f$ are equimeasurable (see Section \ref{S1}). Following \cite[\S2.a]{LT2}, in what follows, we assume that a r.i.\ space is separable or has the Fatou property. Moreover, for a r.i.\ space $X$ the normalization condition $\|\chi_{[0,1]}\|_X=1$ will be assumed.

The K\"othe dual space $X'$ for a r.i.\ space $X$ is again a r.i.\ space, and $X^*=X'$ if and only if $X$ is separable. Every r.i.\ space $X$ on $[0,1]$ (resp. $(0,\infty)$) satisfies  the embeddings $$
L^\infty[0,1]\stackrel{1}{\subseteq} X\stackrel{1}{\subseteq} L^1[0,1]$$
(resp. 
$$
(L^1\cap L^\infty)(0,\infty)\stackrel{1}{\subseteq} X \stackrel{1}{\subseteq} (L^1+L^\infty)(0,\infty))
$$
(if $X$, $Y$ are r.i.\ spaces and $C>0$, then the notation $X\stackrel{C}{\subseteq} Y$ means that this embedding is continuous and $\|x\|_{Y}\le C\|x\|_{X}$, for $x\in X$).

The {\it fundamental function} of $X$ is defined by $\phi_X(t):=\|\chi_A\|_X$, where $A$ is a measurable set with $m(A)=t$.

The family of r.i.\ spaces includes many classical spaces 
appearing in analysis, in particular, $L^p$-spaces, Orlicz spaces,
Lorentz spaces and many others.

Let ${N}$ be an Orlicz function on $[0,\infty)$, i.e., $N$ is a convex continuous increasing function on $[0,\infty)$ with ${N}(0)=0$ and ${N}(\infty)=\infty$. The {\it Orlicz space} $L_N(I)$ consists of all measurable functions $x(t)$ on $I$ for which the Luxemburg norm 
$$
\|x\|_{L_N}:=\inf\Big\{u
>0\,:\,\int_I N(|x(t)|/u) \,dt\leq 1\Big\}$$ 
is finite (see \cite{KR}, \cite{M-89}, \cite{RR}).
In particular, if $N(s)=s^p$, $1\le p<\infty$, we obtain the space $L^p$ with the usual norm. Every Orlicz space $L_N(I)$ has the Fatou property; $L_N[0,1]$ (resp. $L_N(0,\infty)$) is separable if and only if the function $N$ satisfies the {\it $\Delta_2^\infty$-condition} (resp. {\it $\Delta_2$-condition}), i.e., $\sup_{u\ge 1}{N(2u)}/{N(u)}<\infty$ (resp.  $\sup_{u>0}{N(2u)}/{N(u)}<\infty$). The fundamental function of $L_N(I)$ can be calculated by the formula: $\phi_{L_N}(t)=1/N^{-1}(1/t)$, $t\in I$, where $N^{-1}$ is the inverse function for $N$.

Another important class of r.i. spaces is formed by the  Lorentz 
spaces. Let $1\le q<\infty$, and let $\psi$ be an increasing concave function on $I$ such that $\psi(0)=0$.
The {\it Lorentz space} $\Lambda_q(\psi):=\Lambda_q(\psi)(I)$ consists of all functions $x(t)$ measurable on $I$ and satisfying the condition:
\begin{equation}
\label{eqLor} 
\|x\|_{\Lambda_q(\psi)}:=\Big(\int_I x^{*}(t)^{q} 
d\psi(t) \Big)^{1/q}<\infty
\end{equation} 
(see \cite{Lo-51}, \cite{KPS}, \cite[p. 121]{LT2}). For every $1\le q<\infty$ and any concave increasing function $\psi$, $\Lambda_q(\psi)$ is a separable r.i. space with the Fatou property and $\phi_{\Lambda_q(\psi)}(t)=\psi(t)^{1/q}$, $t\in I$.

\subsection{Shift exponents of Banach sequence lattices and dilation indices of r.\ i. function spaces}
\label{prel3}


Let $E$ be a Banach sequence lattice modelled on $\mathbb{Z}$ such that the shift operator $\tau_na:=(a_{k-n})_{k\in\Z}$, where $a=(a_{k})_{k\in\Z}$, is bounded in $E$ for every $n\in\mathbb{Z}$. 
Then, denoting $\Z_+=\{k\in\Z:\,k\ge 0\}$ and $\Z_-=\{k\in\Z:\,k\le 0\}$, for each $n\in\Z$ we set $\tau_{n}^0 a:=\chi_{\Z_-}\cdot\tau_n(a\chi_{\Z_-})$ and $\tau_{n}^\infty a:=\chi_{\Z_+}\cdot\tau_n(a\chi_{\Z_+})$.
Since the norms $\|\tau_{n}\|_{E\to E}$, $\|\tau_{n}^0\|_{E\to E}$ and $\|\tau_{n}^\infty\|_{E\to E}$ are subadditive in $n$, we can define the shift exponents of $E$ by
\begin{align*}
\gamma_E: &= -\lim_{n\to\infty}\frac1n\log_2\|\tau_{-n}\|_{E\to E}, 
&\delta_E: &=\lim_{n\to\infty}\frac{1}{n}\log_2 \|\tau_{n}\|_{E\to E},
\nonumber\\
\gamma_E^0: &= -\lim_{n\to\infty}\frac{1}{n}\log_2 \|\tau_{-n}^0\|_{E\to E}, 
&\delta_E^0&=\lim_{n\to\infty}\frac{1}{n}\log_2 \|\tau_{n}^0\|_{E\to E},
\nonumber\\ \qquad
\gamma_E^\infty&= -\lim_{n\to\infty}\frac{1}{n}\log_2 \|\tau_{-n}^\infty\|_{E\to E}, 
&\delta_E^\infty &=\lim_{n\to\infty}\frac{1}{n}\log_2 \|\tau_{n}^\infty\|_{E\to E}.
\end{align*}

Next, we introduce the definitions of the dilation indices of r.i.\ function spaces. First, consider r.i.\ spaces on $[0,1]$. For any $\tau>0$, the {\it dilation operator} $\tilde{\sigma}_\tau x(t):=x(t/\tau)$, for $0\le t\le \min\{1,1/\tau\}$ and zero elsewhere, is bounded in any r.i. space $X$ on $[0,1]$ and $\|\tilde{\sigma}_\tau\|_{X\to X}\le \max(1,\tau)$, $\tau>0$; see e.g. \cite{Boyd} or \cite[Theorem 2.4.4]{KPS}. The numbers
$$
\alpha_X = \lim_{\tau\to 0+}\frac{\log_2\|\tilde{\sigma}_{\tau}\|_{X\to X}}{\log_2 \tau}\quad\mbox{and}\quad  
\beta_X=\lim_{\tau\to \infty}\frac{\log_2\|\tilde{\sigma}_{\tau}\|_{X\to X}}{\log_2 \tau}
$$
are called the {\it lower and upper Boyd indices} of $X$\footnote{In several places in the literature (see e.g. \cite[p.~131]{LT2}) the dilation indices of a r.i.\ space are taken to be the reciprocals of $\alpha_X$ and $\beta_X$.}. 
Then, $0\le\alpha_X\le \beta_X\le1$ (see, for instance, \cite[\S\,II.4]{KPS}). Equivalently, we have
$$
\alpha_X = -\lim_{n\to\infty}\frac{1}{n}\log_2
\|\tilde{\sigma}_{2^{-n}}\|_{X\to X}\quad\mbox{and}\quad  
\beta_X=\lim_{n\to\infty}\frac{1}{n}\log_2 \|\tilde{\sigma}_{2^n}\|_{X\to X}.
$$

In the case of r.i.\ spaces on $(0,\infty)$ we set ${\sigma}_\tau x(t):=x(t/\tau)$, $t>0$. As above, for an arbitrary r.i.\ space $X$ on $(0,\infty)$, we have $\|{\sigma}_\tau\|_{X\to X}\le \max(1,\tau)$, $\tau>0$, and define the {\it Boyd indices} of $X$ by 
$$
\alpha_X = -\lim_{n\to\infty}\frac{1}{n}\log_2
\|{\sigma}_{2^{-n}}\|_{X\to X}\quad\mbox{and}\quad  
\beta_X=\lim_{n\to\infty}\frac{1}{n}\log_2 \|{\sigma}_{2^n}\|_{X\to X}.$$
Moreover, in the case of r.i.\ spaces on $(0,\infty)$ we will use also the so-called {\it partial} dilation indices. First, for every $\tau>0$ and $x\in L^0(0,\infty)$ we set $\sigma_\tau^0x:=\chi_{[0,1]}\sigma_\tau(x\chi_{[0,1]})$ and 
\begin{align*}
\alpha_X^0:&= -\lim_{n\to\infty}\frac{1}{n}\log_2 \|\sigma_{2^{-n}}^0\|_{X\to X}, 
&\beta_X^0: &=\lim_{n\to\infty}\frac{1}{n}\log_2 \|\sigma_{2^n}^0\|_{X\to X}.
\end{align*}
It is easy to see that $\alpha_X^0={\alpha}_{X[0,1]}$ and $\beta_X^0={\beta}_{X[0,1]}$, where $X[0,1]$ is the  r.i.\ space on $[0,1]$, obtained by restriction of $X$ to $[0,1]$, i.e., 
$$
X[0,1]:=\{f\in X:\, {\rm supp}\,f\subset [0,1]\},\;\;\mbox{with}\;\;\|f\|_{X[0,1]}:=\|f\|_X.$$ 
Secondly, for a  r.i.\ space $X$ on $(0,\infty)$ we denote by ${\mathcal G}_X$ the set of all functions $f\in X$ of the form  $f=c\chi_{[1,2]}+g$, where $c>0$, ${\rm supp}\,g\subset (2,\infty)$ and $|g|\le c$. 
Furthermore, for every integer $n\le 0$ we put ${\mathcal G}_X^n:=\{f\in X:\, \sigma_{2^{n}}f\in {\mathcal G}_X\}$. 
Then, if $\sigma_{2^{n}}^\infty$ is the restriction of the operator $\sigma_{2^{n}}$ to the set ${\mathcal G}_X^n$ if $n\le 0$ and to the set ${\mathcal G}_X$ if $n\ge 0$,
we define 
\begin{align*}
\alpha_X^\infty:&= -\lim_{n\to\infty}\frac{1}{n}\log_2 \|\sigma_{2^{-n}}^\infty\|_{X\to X}, 
&\beta_X^\infty: &=\lim_{n\to\infty}\frac{1}{n}\log_2 \|\sigma_{2^n}^\infty\|_{X\to X}.
\end{align*}
Clearly, $0\le \alpha_X\le\alpha_X^0\le \beta_X^0\le\beta_X\le 1$ 
and $0\le \alpha_X\le\alpha_X^\infty\le \beta_X^\infty\le\beta_X\le 1$.

Let $\psi$ be a positive function on $(0,1]$. Then, the {\it dilation function} $\tilde{M}_\psi$ and the {\it dilation indices} $\mu_\psi$ and $\nu_\psi$ are defined as follows:
$$
\tilde{M}_\psi(t):=\sup_{0<s\le\min(1,1/t)}\frac{\psi(ts)}{\psi(s)}$$
and
\begin{align*}
\mu_\psi&= -\lim_{n\to\infty}\frac{1}{n}\log_2 \tilde{M}_\psi(2^{-n}), 
&\nu_\psi&= \lim_{n\to\infty}\frac{1}{n}\log_2 \tilde{M}_\psi(2^{n}).\end{align*}

For a positive function $\psi$ on $(0,\infty)$, we define three   dilation functions by
$$
M_\psi(t):=\sup_{s>0}\frac{\psi(ts)}{\psi(s)},\quad
M_\psi^0(t):=\sup_{0<s\le\min(1,1/t)}\frac{\psi(ts)}{\psi(s)},
\quad
M_\psi^{\infty}(t):=\sup_{s\ge\max(1,1/t)}\frac{\psi(ts)}{\psi(s)},
$$
and six dilation indices by
\begin{align*}
\mu_\psi&= -\lim_{n\to\infty}\frac{1}{n}\log_2 M_\psi(2^{-n}), 
&\nu_\psi&= \lim_{n\to\infty}\frac{1}{n}\log_2 M_\psi(2^{n}), 
\nonumber\\ 
\mu_\psi^0&= -\lim_{n\to\infty}\frac{1}{n}\log_2 M_\psi^0(2^{-n}), 
&\nu_\psi^0&= \lim_{n\to\infty}\frac{1}{n}\log_2 M_\psi^0(2^{n}), 
\nonumber\\ 
\mu_\psi^\infty&= -\lim_{n\to\infty}\frac{1}{n}\log_2 M_\psi^\infty(2^{-n}), 
&\nu_\psi^\infty&= \lim_{n\to\infty}\frac{1}{n}\log_2 M_\psi^\infty(2^{n}). 
\end{align*}
In the case when $\psi$ is quasi-concave (that is, $\psi(0)=0$, $\psi$ is nondecreasing and $\psi(t)/t$ is nonincreasing), we have $0\leq\mu_\psi\leq \nu_\psi\le 1$ (in the case of $[0,1]$) and $0\leq\mu_\psi\leq\mu_\psi^0\le \nu_\psi^0\le \nu_\psi\le 1$, $0\leq\mu_\psi\leq\mu_\psi^\infty\le \nu_\psi^\infty\le \nu_\psi\le 1$ (in the case of $(0,\infty)$). In particular, the fundamental function $\phi_X$ of a r.i. space $X(I)$ is quasi-concave on $I$. One can easily check that from the above 
definitions it follows that $\alpha_X\le \mu_{\phi_X}$, $ \nu_{\phi_X}\le\beta_X$ (in the case of $[0,1]$) and $\alpha_X\le \mu_{\phi_X}$, $\alpha_X^0\le \mu_{\phi_X}^0$, $\alpha_X^\infty\le \mu_{\phi_X}^\infty$, $ \nu_{\phi_X}\le\beta_X$, 
$\nu_{\phi_X}^0\le\beta_X^0$, $ \nu_{\phi_X}^\infty\le\beta_X^\infty$ (in the case of $(0,\infty)$). 

\begin{definition}
\label{main def}
A r.i.\ space $X$ on $[0,1]$ (resp. on $(0,\infty)$) is said to be of {\it fundamental type} whenever $\alpha_X=\mu_{\phi_X}$ and $\beta_X-\nu_{\phi_X}$ (resp.
\begin{equation*}
\alpha_X=\mu_{\phi_X}\;,\;\alpha_X^0=\mu_{\phi_X}^0\;,\;\alpha_X^\infty
=\mu_{\phi_X}^\infty\;,\;\beta_X=\nu_{\phi_X}\;,\;\beta_X^0
=\nu_{\phi_X}^0\;,\;\beta_X^\infty=\nu_{\phi_X}^\infty).
\end{equation*}
\end{definition}

The most known and important r.i. spaces, in particular, all Lorentz and Orlicz spaces, are of fundamental type. The first example of a r.i. space of non-fundamental type has been constructed by Shimogaki, \cite{Shimo}.


%
%

For a detailed information related to r.i. spaces and their Boyd indices we refer to the books \cite{BSh,KPS,LT2}.

Given two positive functions (quasinorms) $P$ and $Q$, we write $P\asymp Q$ if there exists a positive constant $C$ that does not depend on the arguments of $P$ and $Q$ such that $C^{-1}P\leq Q\leq CP$. Finally, by ${\rm supp}\,f$ we denote the support of a function $f$, i.e., the set $\{t:\,f(t)\ne 0\}$.


\section{Auxiliary results}


\begin{lemma}
\label{indices for fund type}
Let $\psi$ be a positive function on $(0,\infty)$. Then  $\mu_\psi=\min(\mu_\psi^0,\mu_\psi^\infty)$ and $\nu_\psi=\max(\nu_\psi^0,\nu_\psi^\infty)$.

Hence, if $X$ is a r.i. space on $(0,\infty)$ of fundamental type, then  $\alpha_X=\min(\alpha_{X}^0,\alpha_{X}^\infty)$ and $\beta_X=\max(\beta_{X}^0,\beta_{X}^\infty)$.
\end{lemma}
\begin{proof}
Since the proof of both equalities for the shift exponents is very similar, we prove only the second one, for $\nu_\psi$. 

Let $t>1$. Representing the dilation function $M_\psi$ in the following way:
\begin{eqnarray*}
M_\psi(t)&=&\max\left(\sup_{0<s\le 1/t}\frac{\psi(st)}{\psi(s)},\sup_{s\ge 1}\frac{\psi(st)}{\psi(s)},\sup_{1/t\le s\le 1}\frac{\psi(st)}{\psi(s)}\right)\\ &=& \max\left(M_\psi^0(t),M_\psi^\infty(t),\sup_{1/t\le s\le 1}\frac{\psi(st)}{\psi(s)}\right),
\end{eqnarray*}
we get
\begin{eqnarray*}
\nu_\psi &=&\lim_{t\to\infty}\frac{\log_2 M_\psi(t)}{\log_2 t}\\&=&\max\left(\lim_{t\to\infty}\frac{\log_2 M_\psi^0(t)}{\log_2 t},\lim_{t\to\infty}\frac{\log_2 M_\psi^\infty(t)}{\log_2 t},\limsup_{t\to\infty}\frac{\log_2 \sup_{1/t\le s\le 1}\frac{\psi(st)}{\psi(s)}}{\log_2 t}\right)\\&=&\max\left(\nu_\psi^0,\nu_\psi^\infty,\limsup_{t\to\infty}\frac{\log_2 \sup_{1/t\le s\le 1}\frac{\psi(st)}{\psi(s)}}{\log_2 t}\right).
\end{eqnarray*}
Thus, it remains only to verify that 
\begin{equation}
\label{eq1-new}
\limsup_{t\to\infty}\frac{\log_2 \sup_{1/t\le s\le 1}\frac{\psi(st)}{\psi(s)}}{\log_2 t}\le\max(\nu_\psi^0,\nu_\psi^\infty).
\end{equation}

We claim that for all $t>1$ and $s=t^{-\lambda}$, where $\lambda\in [0,1],$ the following formula holds:
\begin{equation}\label{eq6}
\frac{\log_2(\frac{\psi(ts)}{\psi(s)})}{\log_2
t}=(1-\lambda)\frac{\log_2(\frac{\psi(t^{1-\lambda})}
{\psi(1)})}{\log_2(t^{1-\lambda})}+\lambda\frac{\log_2(\frac{\psi(t^\lambda
t^{-\lambda})}{\psi(t^{-\lambda})})}{\log_2(t^{\lambda})}.
\end{equation}
In fact,
\begin{eqnarray*}
\frac{\log_2(\frac{\psi(ts)}{\psi(s)})}{\log_2
t}&=&\frac{\log_2(\frac{\psi(ts)}{\psi(1)})+\log_2(\frac{\psi(1)}{\psi(s)})}
{\log_2 t}\\& =&
(1-\lambda)\frac{\log_2(\frac{\psi(t^{1-\lambda})}{\psi(1)})}{\log_2(t^{1-\lambda})}+\lambda\frac{\log_2(\frac{\psi(t^\lambda
t^{-\lambda})}{\psi(t^{-\lambda})})}{\log_2(t^{\lambda})},
\end{eqnarray*}
and \eqref{eq6} is proved.


Next, for every $t>1$ we choose $s(t)\in [1/t,1]$ so that
$$
\sup_{s\in[1/t,1]}\frac{\psi(ts)}{\psi(s)}=\frac{\psi(ts(t))}{\psi(s(t))}.$$
Then $s(t)=t^{-\lambda(t)}$, where $0\leq\lambda(t)\leq 1$. Let $\{t_n\}$ be any sequence such that $t_n>1$ and $\lim_{n\to\infty}t_n=\infty$. Then, we can assume that the one of the following three conditions is fulfilled:  (a) $\lim_{n\to\infty}t_n^{\lambda(t_n)}=\infty$ and $\lim_{n\to\infty}t_n^{1-\lambda(t_n)}=\infty$; (b) $\lim_{n\to\infty}t_n^{\lambda(t_n)}=\infty$ and $\{t_n^{1-\lambda(t_n)}\}$ is bounded from above; (c)  $\lim_{n\to\infty}t_n^{1-\lambda(t_n)}=\infty$ and $\{t_n^{\lambda(t_n)}\}$ is bounded from above.

If (a) holds, then formula \eqref{eq6} implies that 
\begin{eqnarray*}
\frac{\log_2(\sup_{1/t\le s\le 1}\frac{\psi(t_ns)}{\psi(s)})}{\log_2
t_n}&\le&(1-\lambda(t_n))\frac{\log_2(\sup_{s\ge 1}\frac{\psi(st^{1-\lambda(t_n)})}
{\psi(s)})}{\log_2(t^{1-\lambda(t_n)})}\\ &+& \lambda(t_n)\frac{\log_2(\sup_{t^{-\lambda(t_n)}\le s\le 1}\frac{\psi(st^{\lambda(t_n)})}{\psi(s)})}{\log_2(t^{\lambda(t_n)})}\\&\le&
\max\left(\frac{\log_2(\sup_{s\ge 1}\frac{\psi(st^{1-\lambda(t_n)})}
{\psi(s)})}{\log_2(t^{1-\lambda(t_n)})},\frac{\log_2(\sup_{t^{-\lambda(t_n)}\le s\le 1}\frac{\psi(st^{\lambda(t_n)})}{\psi(s)})}{\log_2(t^{\lambda(t_n)})}\right),
\end{eqnarray*}
and hence
$$
\limsup_{n\to\infty}\frac{\log_2(\sup_{1/t\le s\le 1}\frac{\psi(t_ns)}{\psi(s)})}{\log_2
t_n}\le \max(\nu_\psi^0,\nu_\psi^\infty).
$$
In the case (b) $\lambda(t_n)\to 1$ as $n\to\infty$. Therefore, from \eqref{eq6} it follows that
$$
\limsup_{n\to\infty}\frac{\log_2(\sup_{1/t\le s\le 1}\frac{\psi(t_ns)}{\psi(s)})}{\log_2
t_n}\le \lim_{n\to\infty}\frac{\log_2(\sup_{t^{-\lambda(t_n)}\le s\le 1}\frac{\psi(st^{\lambda(t_n)})}{\psi(s)})}{\log_2(t^{\lambda(t_n)})}=\nu_\psi^0.
$$
Similarly, if (c) holds, then $\lambda(t_n)\to 0$ as $n\to\infty$, which implies that
$$
\limsup_{n\to\infty}\frac{\log_2(\sup_{1/t\le s\le 1}\frac{\psi(t_ns)}{\psi(s)})}{\log_2
t_n}\le \lim_{n\to\infty}\frac{\log_2(\sup_{s\ge 1}\frac{\psi(st^{1-\lambda(t_n)})}
{\psi(s)})}{\log_2(t^{1-\lambda(t_n)})}=\nu_\psi^\infty.
$$
Summarizing all, we get inequality \eqref{eq1-new} and thereby arrive at the desired result.

The second assertion of the lemma related to the dilation indices of a r.i.\ space of fundamental type is a straightforward consequence of the first one and Definition \ref{main def}.

\end{proof}

\begin{lemma}
\label{separation}
Suppose $X$ is a r.i. space on $(0,\infty)$ such that
$$
\|x\|_X\asymp \max(\|x^*\chi_{[0,1]}\|_Y,\|x\|_Z),$$
where $Y$ and $Z$ are r.i. spaces on $[0,1]$ and $(0,\infty)$, respectively. Then, 

(a) 
$\phi_X(t)\asymp\max(\phi_Y(t),\phi_Z(t))$ for $0<t\le 1$ and $\phi_X(t)=\phi_Z(t)$ for $t>1$;

(b) $\alpha_X^0=\max(\alpha_{Y},\alpha_{Z}^0)$ and $\beta_X^0=\max(\beta_{Y},\beta_{Z}^0)$;

(c) $\alpha_X^\infty=\alpha_{Z}^\infty$ and $\beta_X^\infty=\beta_{Z}^\infty$.
\end{lemma}
\begin{proof}
Since the assertions (a) and (b) are obvious, it suffices to prove  (c). 


Let $\tau>1$ and $x\in\mathcal{G}_X$. By homogeneity, we can assume that $x=\chi_{[1,2]}+y$, where ${\rm supp}\,y\subset (2,\infty)$ and $|y|\le 1$. Then, by the hypothesis, it follows
$$
\|x\|_X\asymp\max(1,\|\chi_{[1,2]}+y\|_Z)=\|x\|_Z\;\;\mbox{and}\;\;\|\sigma_{{\tau}}x\|_X\asymp\max(1,\|\sigma_{{\tau}}x\|_Z)=\|\sigma_{{\tau}}x\|_Z.$$
Consequently, 
$$
\|\sigma_{{\tau}}^\infty\|_{X\to X}=\|\sigma_{{\tau}}\|_{\mathcal{G}_X\to X}=\sup_{x\in\mathcal{G}_X}\frac{\|\sigma_{{\tau}}x\|_X}{\|x\|_X}=\sup_{x\in\mathcal{G}_Z}\frac{\|\sigma_{{\tau}}x\|_Z}{\|x\|_Z}=\|\sigma_{{\tau}}^\infty\|_{Z\to Z}.
$$
Thus, $\beta_X^\infty=\beta_{Z}^\infty$. The equality $\alpha_X^\infty=\alpha_{Z}^\infty$ can be obtained in the same way, so we skip this proof. 
\end{proof}

Let $X$ be a r.i.\ space on $[0,1]$. We introduce the r.i.\ space $X_1$ on $(0,\infty)$ given by the norm 
%
\begin{equation}
\label{new space}
\|x\|_{X_1}:=\max\left(\|x^*\chi_{[0,1]}\|_X,\|x\|_{L^1(0,\infty)}\right).
\end{equation}

\begin{lemma}\label{p-5}
If $X$ is a r.i.\ space on $[0,1]$ of fundamental type, so is $X_1$.
\end{lemma}
\begin{proof}
Thanks to Lemma \ref{separation}, we need to prove only that $\beta_{X_1}=\nu_{\phi_{X_1}}$ and $\alpha_{X_1}=\mu_{\phi_{X_1}}$. The first equality is immediate. Indeed, by the same lemma, we have $\phi_{X_1}(t)\asymp t$, for $t>1$, and hence $\nu_{\phi_{X_1}}=1$. It remains to note that,  for an arbitrary r.i.\ space $Y$, we have $ \nu_{\phi_{Y}}\le\beta_Y\le 1$. 

%

The equality $\alpha_{X_1}=\mu_{\phi_{X_1}}$ is a little bit more delicate. Let $n>0$ and $x\in X_1$ be any nonincreasing nonnegative function such that $\|x\|_{X_1}=1$. Without loss of generality, we may assume that $x(1)>0$. First, we introduce the function $y$ by
$$
y(t)=x(t)\chi_{[0,1]}(t)+x(1)\chi_{[1,a+1]}(t),$$
where $a:=\frac{1}{x(1)}\int_1^\infty x(s)\,ds$. Clearly, $y$ is a nonincreasing nonnegative function on $(0,\infty)$ and $\|y\|_{X_1}=\|x\|_{X_1}=1$. In particular, this implies that $x(1)\le 1$ and $\int_1^\infty x(s)\,ds\le 1$. In consequence
\begin{equation}
\label{equalities for ind}
x(1)\le \frac{1}{a}\;\;\mbox{and}\;\;a\ge 1.
\end{equation}

Next, from the definition of $y$ it follows that
$$
\int_0^t x(2^{-n}s)\,ds\le \int_0^t y(2^{-n}s)\,ds,\;\;t>0.$$
Consequently, since $X$ is separable or has the Fatou property (see Section \ref{prel2}), applying \cite[Proposition~2.a.8]{LT2}, we have 
\begin{equation}
\label{Hardy type}
\|(\sigma_{2^{-n}}x)\chi_{(0,1]}\|_{X}\le \|(\sigma_{2^{-n}}y)\chi_{(0,1]}\|_{X}.
\end{equation}
Moreover, by using the first inequality from \eqref{equalities for ind}, we get
\begin{eqnarray*}
\|(\sigma_{2^{-n}}y)\chi_{(0,1]}\|_{X}&\le& \|\sigma_{2^{-n}}(x\chi_{(0,2^{-n}]})\|_{X}+x(1)\|\chi_{(2^{-n},\min(1,2^{-n}(a+1))}\|_X\\ &\le& \|\tilde{\sigma}_{2^{-n}}\|_{X\to X}\|x\chi_{(0,1)}\|_{X}+\frac{1}{a}\phi_X(\min(1,2^{-n}a)).
\end{eqnarray*}
Observe that $\phi_X(\min(1,2^{-n}a))=\phi_X(1)=1$ if $a\ge 2^n$, and, since  $\phi_X$ is quasi-concave, by the second inequality from \eqref{equalities for ind}, 
$$\phi_X(\min(1,2^{-n}a))=\phi_X(2^{-n}a)\le a\phi_X(2^{-n})\le aM_{\phi_X}(2^{-n})$$ 
if $a< 2^n$.
Therefore, summarizing all, we obtain 
$$
\|(\sigma_{2^{-n}}y)\chi_{(0,1]}\|_{X}\le\max\left(\|\tilde{\sigma}_{2^{-n}}\|_{X\to X},M_{\phi_X}(2^{-n}),2^{-n}\right)=\|\tilde{\sigma}_{2^{-n}}\|_{X\to X}.
$$
Combining this together with \eqref{Hardy type} and the equality $\|\sigma_{2^{-n}}\|_{L^1\to L^1}=2^{-n}$, we deduce that
$$
\|\sigma_{2^{-n}}\|_{X_1\to X_1}\le \|\tilde{\sigma}_{2^{-n}}\|_{X\to X}.$$
Hence, since $X$ is of fundamental type, we conclude
$$
\alpha_{X_1}\ge \alpha_X=\mu_{\phi_X}\ge \mu_{\phi_{X_1}}.$$
Since the opposite inequality is immediate, everything is done.

\end{proof}


\section{Connections between of dilation indices of a r.i.\ space on $(0,\infty)$ and shift exponents of a suitable Banach sequence lattice}


Following \cite{A-22p} (see also \cite{Kal}), we assign to every r.i. function space $X$ on $(0,\infty)$ a certain Banach sequence lattice $E_X$ such that the sequence $\{\chi_{\Delta_k}\}_{k\in\Z}$, where $\Delta_k:=[2^k,2^{k+1})$, is equivalent in $X$ to the unit vector basis $\{e_k\}_{k\in\Z}$ in $E_X$. 

Let $X$ be a r.i. space on $(0,\infty)$. For an arbitrary sequence
$a=(a_k)_{k\in\Z}$ we introduce the following step function
$$
Sa(t):=\sum_{k\in\Z} a_k\chi_{\Delta_k}(t),\;\;t>0.
$$ 
We associate to $X$ the Banach sequence lattice $E_X$ equipped with the norm
$$
\Big\|\sum_{k\in\Z} a_ke_k\Big\|_{E_X}:=\|Sa\|_X.$$

A crucial role in the proof of Theorem \ref{t-1} is played by properties of the 
shift exponents of the Banach sequence lattice $E_X$ (see \cite{A-22p}). 
In this section, we present a full proof of a  refined version of Lemma 2 
from \cite{A-22p}, which was proved there only in 
part\footnote{We correct here a certain inaccuracy in the 
definition of the indices $\alpha_X^\infty$ and $\beta_X^\infty$ in the paper \cite{A-22p}.}.

\begin{proposition}
\label{dilation}
For every r.i. space $X$ on $(0,\infty)$ and all $n\in\mathbb{Z}$ we have: 

(i) $\|\tau_n\|_{E_X\to E_X}\le \|\sigma_{2^n}\|_{X\to X}\le 2 \|\tau_n\|_{E_X\to E_X}$;

(ii) $\|\tau_{n}^0\|_{E_X\to E_X}\le\|\sigma_{2^n}^0\|_{X\to X}\le 2\|\tau_{n}^0\|_{E_X\to E_X}$;

(iii) $\frac 12\|\tau_{n}^\infty\|_{E_X\to E_X}\le \|\sigma_{2^n}^\infty\|_{X\to X}\le 4\|\tau_{n}^\infty\|_{E_X\to E_X}$.

Hence, $\alpha_X=\gamma_{E_X}$, $\alpha_X^0=\gamma_{E_X}^0$, $\alpha_X^\infty=\gamma_{E_X}^\infty$, $\beta_X=\delta_{E_X}$, $\beta_X^0=\delta_{E_X}^0$ and $\beta_X^\infty=\delta_{E_X}^\infty$.
\end{proposition}
\begin{proof}

Part (i) is proved in \cite[Lemma 2]{A-22p}.

(ii) By the definition of the operator $\tau_{n}^0$, for every $n\in\mathbb{Z}$ and any $a=(a_{k})_{k\in\Z}$, we have
$\tau_{n}^0a=\sum_{k\le\min(0,n)} a_ke_k$. Therefore,
\begin{equation*}
S(\tau_n^0a)= \sum_{k\le\min(0,n)} a_{k-n}\chi_{\Delta_k}=\sum_{j\le\min(0,-n)} a_{j}\chi_{\Delta_{n+j}}=\sigma_{2^n}(Sa^{(n,-)}),
\end{equation*}
where 
\begin{equation}
\label{oper-minus}
a^{(n,-)}:=\sum_{j\le \min(0,-n)}a_{j}e_j.
\end{equation}

Since ${\rm supp\,}Sa^{(n,-)}\cup{\rm supp\,}\sigma_{2^n}(Sa^{(n,-)})\subset [0,1]$, from the definition of the operator $\sigma_{2^n}^0$ (see Section \ref{prel3}) it follows that for every $n\in\mathbb{Z}$ and any $a=(a_{k})_{k\in\Z}$ 
\begin{equation}
\label{equa102}
S(\tau_n^0a)=\sigma_{2^n}^0(Sa^{(n,-)}).
\end{equation}
Hence, in view of the inequality $|a^{(n,-)}|\le |a|$, we get
\begin{eqnarray*}
\|\tau_n^0a\|_{E_X}&=&\|S(\tau_n^0a)\|_X=\|\sigma_{2^n}^0(Sa^{(n,-)})\|_X\le \|\sigma_{2^n}^0\|_{X\to X}\|Sa^{(n,-)}\|_X\\ &=&\|\sigma_{2^n}^0\|_{X\to X}\|a^{(n,-)}\|_{E_X}\le \|\sigma_{2^n}^0\|_{X\to X}\|a\|_{E_X}.
\end{eqnarray*}
As a result, we conclude that $\|\tau_{n}^0\|_{E_X\to E_X}\le\|\sigma_{2^n}^0\|_{X\to X}$. In particular, since $\|\sigma_{2}^0\|_{X\to X}\le \|\sigma_{2}\|_{X\to X}\le 2$, the last inequality implies that
\begin{equation}
\label{estimate of tau-oper}
\|\tau_{1}^0\|_{E_X\to E_X}\le 2.
\end{equation}


Before proving the opposite inequality, we define the  averaging operator $Q$ by 
\begin{equation}\label{projection Q}
Qx(t):=\sum_{k\in\Z} 2^{-k}\int_{\Delta_k}x(s)\,ds\cdot \chi_{\Delta_k}(t),\;\;t>0.
\end{equation}
It is well known that $Q$ is a one norm projection on each r.i. space $X$; see e.g. \cite[\S\,II.3.2]{KPS} (recall that $X$ is assumed to be separable or to have the Fatou property; see Section \ref{prel2}).
Let $x\in X$. Setting 
$$
a_x:=\left(2^{-k}\int_{\Delta_k}x(s)\,ds\right)_{k\in\Z}$$ 
and comparing the operators $S$ and $Q$, one can see that 
\begin{equation}\label{equal of oper}
Sa_x=Qx.
\end{equation}
Moreover, we have 
\begin{equation}\label{dilation1}
a_{\sigma_2x}=\sum_{k\in\Z} 2^{-k}\int_{\Delta_k}x(s/2)\,ds\cdot e_k=\sum_{k\in\Z} 2^{-(k-1)}\int_{\Delta_{k-1}}x(s)\,ds\cdot e_{k}=\tau_1a_x.
\end{equation}

Let $x\in X$, ${\rm supp}\,x\subset [0,1]$. Without loss of generality, we can assume that $x=x^*$. Then, it can be easily checked that $x(t)\le Q\sigma_2^0x(t)$, $t>0$. In addition, 
in this case $a_x=a_x\chi_{\Z_-}$
and from \eqref{dilation1} and \eqref{oper-minus} it follows that $a_{\sigma_2^0x}=\tau_1^0a_x^{(1,-)}$.
Thus, by \eqref{equal of oper}, 
$$
\|\sigma_{2^n}^0x\|_X \le \|\sigma_{2^n}^0(Q\sigma_2x)\|_X=\|\sigma_{2^n}^0(S a_{\sigma_2x})\|_X=\|\sigma_{2^n}^0(S \tau_1^0 a_x^{(1,-)})\|_X.$$
Furthermore, in view of the definition of the operator $\sigma_{2^n}^0$ and \eqref{oper-minus}, we have
$$
\sigma_{2^n}^0(S \tau_1^0 a_x^{(1,-)})=\sigma_{2^n}^0(S \tau_1^0 a_x^{(\max(0,n)+1,-)}).$$
Therefore, by \eqref{equa102}, \eqref{estimate of tau-oper} and the inequality $|a_x^{(m,-)}|\le |a_x|$, $m\in\Z$, it holds that
\begin{eqnarray*}
\|\sigma_{2^n}^0x\|_X &\le& \|\sigma_{2^n}^0(S \tau_1^0 a_x^{(\max(0,n)+1,-)})\|_X\le\|S( \tau_{n+1}^0 a_x^{(\max(0,n)+1,-)})\|_X\\ &\le&
\|\tau_{n+1}^0a_x\|_{E_X}\le \|\tau_{n+1}^0\|_{E_X\to E_X}\|Qx\|_{X}\le 2\|\tau_{n}^0\|_{E_X\to E_X}\|x\|_{X}.
\end{eqnarray*}
Thus, $\|\sigma_{2^n}^0\|_{X\to X}\le 2\|\tau_{n}^0\|_{E_X\to E_X}$, and (ii) is proved.

(iii) Let $n\in\mathbb{Z}$. Given $a=(a_{k})_{k\in\Z}$, we have
$\tau_{n}^\infty a=\sum_{k\ge\max(0,n)}a_ke_k$. Hence,
\begin{equation}
\label{equa102a}
S(\tau_n^\infty a)= \sum_{k\ge\max(0,n)} a_{k-n}\chi_{\Delta_k}=\sum_{j\ge\max(0,-n)} a_{j}\chi_{\Delta_{n+j}}=\sigma_{2^n}(Sa^{(n,+)}),
\end{equation}
where $a^{(n,+)}:=\sum_{j\ge \max(0,-n)}a_{j}e_j$. 
Observe that ${\rm supp\,}Sa^{(n,+)}\cup{\rm supp\,}\sigma_{2^n}(Sa^{(n,+)})\subset (1,\infty)$. However, in general, the function $Sa^{(n,+)}$ does not belong to the set $\mathcal{G}_X$. So, we need slightly to change it. Let 
$$
b:= \|a^{(n,+)}\|_{\ell^\infty}e_{\max(0,-n)}+a^{(n,+)}\chi_{\{k\ge\max(1,-n+1)\}}.$$
Then, $Sb\in\mathcal{G}_X$ if $n\ge 0$ and  $Sb\in\mathcal{G}_X^n$ if  $n<0$. Moreover, $|a^{(n,+)}|\le |b|$ and hence from \eqref{equa102a} it follows that
\begin{equation*}
\label{equa102a-new}
S(\tau_n^\infty a)\le \sigma_{2^n}^\infty(Sb).
\end{equation*}
Combining this estimate with the inequalities
$$
\|e_{\max(0,-n)}\|_{E_X}=\|\chi_{\Delta_{max(0,-n)}}\|_X\le\frac{1}{\|a^{(n,+)}\|_{\ell^\infty}}\Big\|\sum_{j\ge \max(0,-n)}a_{j}\chi_{\Delta_j}\Big\|_X=\frac{\|a^{(n,+)}\|_{E_X}}{\|a^{(n,+)}\|_{\ell^\infty}}$$
and $|a^{(n,+)}|\le |a|$, we obtain 
\begin{eqnarray*}
\|\tau_n^\infty a\|_{E_X}&=&\|S(\tau_n^\infty a)\|_X\le \|\sigma_{2^n}^\infty(Sb)\|_X\le
\|\sigma_{2^n}^\infty\|_{X\to X}\|Sb\|_X\\&\le& \|\sigma_{2^n}^\infty\|_{X\to X}\|S(\|a^{(n,+)}\|_{\ell^\infty}e_{\max(0,-n)}+a^{(n,+)}\chi_{\{k\ge\max(1,-n+1)\}})\|_X
\\&\le& \|\sigma_{2^n}^\infty\|_{X\to X}(\|a^{(n,+)}\|_{\ell^\infty}\|e_{\max(0,-n)}\|_{E_X}+\|a^{(n,+)}\|_{E_X})
\\ &\le& 2\|\sigma_{2^n}^\infty\|_{X\to X}\|a^{(n,+)}\|_{E_X}\le 2\|\sigma_{2^n}^\infty\|_{X\to X}\|a\|_{E_X}.
\end{eqnarray*}
Thus, $\|\tau_n^\infty\|_{E_X\to E_X}\le 2\|\sigma_{2^n}^\infty\|_{X\to X}$. Since $\|\sigma_{2}^\infty\|_{X\to X}\le \|\sigma_{2}\|_{X\to X}\le 2$, the latter inequality implies that
\begin{equation}
\label{estimate of tau-oper2}
\|\tau_{1}^\infty\|_{E_X\to E_X}\le 2.
\end{equation}


It remains to prove the opposite inequality. Let $n\in\Z$ and let
$x$ be any function from the set $\mathcal{G}_X$ if $n\ge 0$ and $\mathcal{G}_X^n$ if $n\le 0$, that is, $x=c\chi_{\Delta_{max(0,-n)}}+y$, where ${\rm supp}\,y\subset (max(2,2^{-n+1}),\infty)$ and $|y|\le c$. It can be assumed also that  $x$ is nonnegative and nonincreasing for $t\ge max(1,2^{-n})$, and, by homogeneity, that $c=1$. Then, setting $x':=\chi_{\Delta_{max(0,-n)}}+\sigma_2 x$, we have $Qx'\ge x$. Indeed, $Qx'=x(t)=0$ for $0< t<max(1,2^{-n})$ and $Qx'(t)=x(t)=1$ for $max(1,2^{-n})\le t\le max(2,2^{-n+1})$. Finally, if $2^k< t\le 2^{k+1}$, where $k\ge\max(1,-n+1)$, we have 
$$
Qx'(t)=2^{-k}\int_{\Delta_k}\sigma_2x(s)\,ds=2^{-k}\int_{2^{k}}^{2^{k+1}}x(s/2)\,ds\ge x(2^k)\ge x(t).$$
Furthermore, from \eqref{dilation1} it follows  
\begin{equation*}
\label{equa102ab}
a_{x'}=e_{max(0,-n)}+\sum_{k\ge 1+max(0,-n)}2^{-(k-1)}\int_{\Delta_{k-1}}x(s)\,ds\cdot e_{k}=e_{max(0,-n)}+\tau_1^\infty a_x^{(1,+)}.
\end{equation*}
Thus, applying successively \eqref{equal of oper}, \eqref{equa102a}, \eqref{estimate of tau-oper2}, the inequality $\chi_{\Delta_{max(0,-n)}}\le x$ and the fact that the operator $Q$ defined in  \eqref{projection Q} is a one norm projection in $X$,  we obtain
\begin{eqnarray*}
\|\sigma_{2^n}^\infty x\|_X &\le& \|\sigma_{2^n}(Qx')\|_X=\|\sigma_{2^n}(S a_{x'})\|_X=\|\sigma_{2^n}(S(e_{max(0,-n)}+\tau_1^\infty a_x^{(1,+)}))\|_X\\&=&
\|S(\tau_{n}^\infty e_{max(0,-n)}+\tau_{n+1}^\infty a_x^{(1,+)})\|_X
=\|\tau_{n}^\infty e_{max(0,-n)}+\tau_{n+1}^\infty a_x^{(1,+)}\|_{E_X}\\
&\le& \|\tau_{n+1}^\infty\|_{E_X\to E_X}(\|\chi_{\Delta_{max(0,-n)}}\|_X+\|Qx\|_{X})\le 4\|\tau_{n}^\infty\|_{E_X\to E_X}\|x\|_{X}.
\end{eqnarray*}
Thus, $\|\sigma_{2^n}^\infty\|_{X\to X}\le 4\|\tau_{n}^\infty\|_{E_X\to E_X}$, and the proof of (iii) is completed.

It remains to note that all the required equalities for the dilation indices of $X$ and the shift indices of $E_X$ follow immediately from the obtained inequalities  for norms of the dilation and shift operators.
\end{proof}


\section{Proof of the main results}

\begin{proof}[Proof of Theorem \ref{t-2}]
First, it can be easily showed that for every r.i.\ space $X$ on $[0,1]$ we have the embedding
\begin{equation}\label{eq-12}
\mathcal{F}(X)\subseteq[1/\beta_X,1/\alpha_X].
\end{equation}
Indeed, assume that $p\in \mathcal{F}(X)$. Then, as an immediate consequence of the definition of the set $\mathcal{F}(X)$ (see \eqref{main1}), for every $m\in\N$, we can find functions $u_m, v_m\in X$, $\|u_m\|_X=\|v_m\|_X=1$, satisfying $\|\tilde{\sigma}_mu_m\|_X\ge \frac 12 m^{1/p}$ and $\|\tilde{\sigma}_{1/m}v_m\|_X\ge \frac 12 m^{-1/p}$. This implies that $\|\tilde{\sigma}_m\|_{X\to X}\ge \frac 12 m^{1/p}$ and $\|\tilde{\sigma}_{1/m}\|_{X\to X}\ge \frac 12 m^{-1/p}$. Then, these inequalities and the definition of the Boyd indices of $X$ imply that $1/\beta_X\le p\le 1/\alpha_X$. Thus, it remains only to prove the opposite embedding
\begin{equation}\label{eq-12}
\mathcal{F}(X)\supseteq[1/\beta_X,1/\alpha_X].
\end{equation}

From now we will assume that $X$ is a separable r.i.\ space on $[0,1]$ of fundamental type. Let $p\in [{1}/{\beta_X},{1}/{\alpha_X}]$. For every $m\in\N$ and 
$\varepsilon\in (0,1)$, we need to find equimeasurable functions $x_k\in X$, $k=1,2,\dots,m$, 
such that $\textrm{supp}\,x_i\cap \textrm{supp}\,x_j=\varnothing$ 
for $i\ne j$ and for any $a_k\in\mathbb{R}$ 
\begin{equation}\label{eq-13}
(1+\varepsilon)^{-1}\|(a_k)_{k=1}^m\|_p\le \Big\|
\sum_{k=1}^ma_kx_k\Big\|_X\le (1+\varepsilon) \|(a_k)_{k=1}^m\|_p.
\end{equation}
%

Suppose first that $p>1$. Consider the r.i.\ space $X_1$ on $(0,\infty)$ defined by formula \eqref{new space}. 
%
%
From Lemmas \ref{p-5} and \ref{separation} it follows that $X_1$ is a separable r.i. space of fundamental type, $\alpha_{X_1}^0=\alpha_{X}$, $\beta_{X_1}^0=\beta_{X}$, $\alpha_{X_1}^\infty=\beta_{X_1}^\infty=1$.
Moreover, applying Lemma \ref{indices for fund type}, we have that
$\alpha_{X_1}=\min(\alpha_{X},1)=\alpha_{X}$ and $\beta_{X_1}=\max(\beta_{X},1)=1$. Therefore, by Theorem \ref{t-1},
\begin{equation}\label{eq-14}
\mathcal{F}(X_1)=\{1\}\cup
\big[{1}/{\beta_X}, {1}/{\alpha_X}\big].
\end{equation}

Let $m\in\N$ and $\varepsilon\in (0,1)$ be fixed. 
For definiteness, we will further assume that $p<\infty$ (the case when 
$p=\infty$ can be considered quite similarly). Set $\eta:=\frac{\varepsilon}{2(1+\varepsilon)}$ and take 
$n\in\N$ such  that 
%
\begin{equation}\label{eq-15}
n>\max\Big\{\Big(\frac{2m}{1-\eta}\Big)^{2p/(p-1)},
\Big(\frac{2m}{\eta}\Big)^{2p/(p-1)}\Big\}.
\end{equation}
By \eqref{eq-14}, for the given $p\in [{1}/{\beta_X},{1}/{\alpha_X}]$,   
there exist equimeasurable functions $f_k\in X_1$, $k=1,2,\dots,n$,  
$\textrm{supp}\,f_i\cap \textrm{supp}\,f_j=\varnothing$ if $i\ne j$, 
such that for any $a_k\in\mathbb{R}$ we have
\begin{equation}\label{eq-16} 
(1-\eta)\|(a_k)_{k=1}^n\|_p\le \Big\|
\sum_{k=1}^na_kf_k\Big\|_{X_1}\le (1+\eta) \|(a_k)_{k=1}^n\|_p.
\end{equation} 
Since $X_1$ is separable, we can assume that $f_k(t)=f(t-(k-1)h)$, $k=1,\dots,n$, where $f$ is a nonincreasing, nonnegative function on $(0,\infty)$, $\textrm{supp}\,f=(0,h)$ for some $h>0$. Moreover, by  \eqref{new space},
$$
\Big\|\sum_{k=1}^nf_k\Big\|_{X_1}=
\max\left(\|\sigma_n(f\chi_{[0,1/n]})\|_X,\|\sigma_nf\|_{L^1(0,\infty)}\right).
$$
Therefore, from  \eqref{eq-16} it follows that 
$$
\|\sigma_nf\|_{L^1(0,\infty)}\le 2n^{1/p},
$$
and, hence,
\begin{equation}\label{eq-17} 
\|f_k\|_{L^1(0,\infty)}=\|f\|_{L^1(0,\infty)}=n^{-1}\|\sigma_nf\|_{L^1(0,\infty)}\le 2n^{(1-p)/p},\;\;k=1,2,\dots,n.
\end{equation} 
In particular, by \eqref{eq-17} and H\"{o}lder inequality, for all $a_k\in\mathbb{R}$, $k=1,\dots,m$, not all of which are zero, we have
$$
\Big\|\sum_{k=1}^ma_kf_k\Big\|_{L^1(0,\infty)} 
\le 2\sum_{k=1}^m|a_k|n^{(1-p)/p}
\le 2\left(\frac{m}{n}\right)^{(p-1)/p}\|(a_k)_{k=1}^m\|_p.
$$
Since the choice of $n$ (see \eqref{eq-15}) ensures that  $2({m}/{n})^{(p-1)/p}< 1-\eta$, we infer that
$$
\Big\|\sum_{k=1}^ma_kf_k\Big\|_{L^1(0,\infty)} <(1-\eta)\|(a_k)_{k=1}^m\|_p.
$$
Thus, from the definition of the norm in $X_1$ and inequality \eqref{eq-16}, for all $a_k\in\mathbb{R}$, it follows
\begin{equation}\label{eq-18} 
(1-\eta)\|(a_k)_{k=1}^m\|_p\le \Big\|
\Big(\sum_{k=1}^ma_kf_k\Big)^*\chi_{[0,1]}\Big\|_{X}\le (1+\eta) \|(a_k)_{k=1}^m\|_p.
\end{equation}

Now, we claim that 
\begin{equation}\label{eq-19} 
f(t)\le 2 n^{(1-p)/(2p)}\;\;\mbox{for}\;\;t>\frac{2n^{(1-p)/(2p)}}{1-\eta}.
\end{equation} 
Indeed, let 
$$
m\big\{t>0:\,f(t)\ge n^{(1-p)/(2p)}\|f\|_{X_1}\big\}\ge\delta.
$$
Then, by \eqref{eq-17} and \eqref{eq-16}, we have
$$
2n^{(1-p)/p}\ge \|f\|_{L^1(0,\infty)}\ge\int_0^\delta f(t)\,dt
\ge n^{(1-p)/(2p)}\delta \|f\|_{X_1}\ge n^{(1-p)/(2p)}\delta(1-\eta),
$$
which implies that
$$
\delta\le  \frac{2n^{(1-p)/(2p)}}{1-\eta}.
$$
Since $f$ is nonincreasing and $\|f\|_{X_1}\le 2$ (see \eqref{eq-16}), we obtain \eqref{eq-19}.

Let the sets $E_k\subset (0,\infty)$ be chosen in such a way that each function $f_k\chi_{E_k}$, $k=1,\dots,m$, is equimeasurable with the function $f\chi_{[0,1/m]}$. Furthermore, suppose that coefficients $a_k$, $k=1,\dots,m$, are fixed. Then, there are pairwise disjoint sets $A_k\subset (0,\infty)$, $m(\cup_{k=1}^m A_k)=1$ (they depend on $a_k$ and some of them may be empty), such that
$$
\Big(\sum_{k=1}^ma_kf_k\Big)^*\chi_{[0,1]}=\Big|\sum_{k=1}^ma_kf_k\chi_{A_k}\Big|.
$$
Denoting $E_k':=A_k\setminus E_k$, $k=1,\dots,m$, we have
\begin{eqnarray}
\Big\|\sum_{k=1}^ma_kf_k\chi_{E_k}\Big\|_{X_1} &\ge&  \Big\|
\sum_{k=1}^ma_kf_k\chi_{A_k}\Big\|_{X_1}
-\Big\|\sum_{k=1}^ma_kf_k\chi_{E_k'}\Big\|_{X_1}\nonumber\\&=&
\Big\|
\Big(\sum_{k=1}^ma_kf_k\Big)^*\chi_{[0,1]}\Big\|_{X}
-\Big\|\sum_{k=1}^ma_kf_k\chi_{E_k'}\Big\|_{X_1}.
\label{est1}
\end{eqnarray}

Observe that from \eqref{eq-15} it follows
$$
\frac{1}{m}\ge\frac{2n^{(1-p)/(2p)}}{1-\eta}.
$$
Therefore, in view of \eqref{eq-19}, 
$f_k(t)\chi_{E_k'}(t)\le 2 n^{(1-p)/(2p)}$ for all $t>0$. Thus, by \eqref{est1}, \eqref{eq-18} and \eqref{eq-15}, we have
\begin{eqnarray*}
\Big\|\sum_{k=1}^ma_kf_k\chi_{E_k}\Big\|_{X_1}
&\ge& (1-\eta)\|(a_k)_{k=1}^m\|_p-2n^{(1-p)/(2p)}m\max_{k=1,\dots,m}|a_k|
\\ &\ge& 
(1-\eta-2n^{(1-p)/(2p)}m)\|(a_k)_{k=1}^m\|_p\ge (1-2\eta)\|(a_k)_{k=1}^m\|_p.
\end{eqnarray*}
Since 
$$
\Big\|\sum_{k=1}^ma_kf_k\chi_{E_k}\Big\|_{X_1}\le \Big\|
\Big(\sum_{k=1}^ma_kf_k\Big)^*\chi_{[0,1]}\Big\|_{X},$$
then, thanks to the choice of $\eta$, the last inequality combined 
with \eqref{eq-18} implies that 
\begin{equation}
(1+\varepsilon)^{-1}\|(a_k)_{k=1}^m\|_p\le \Big\|\sum_{k=1}^ma_kf_k\chi_{E_k}\Big\|_{X_1}\le (1+\varepsilon) \|(a_k)_{k=1}^m\|_p.
\label{in X1}
\end{equation}
Recall that the functions $f_k\chi_{E_k}$, $k=1,2,\dots,m$, are pairwise disjoint and equimeasurable with the function $f\chi_{[0,1/m]}$. Moreover, their choice does not depend on coefficients $a_k$, $k=1,2,\dots,m$. Taking for $x_k$ the translates of  $f\chi_{[0,1/m]}$ to the intervals $((k-1)/m,k/m)$, i.e., $x_k(t):=f(t-(k-1)/m)\chi_{(k-1)/m,k/m)}$, $k=1,2,\dots,m$, we see that these are pairwise disjoint and equimeasurable functions from $X$. Also, since the functions $\sum_{k=1}^ma_kf_k\chi_{E_k}$ and $\sum_{k=1}^ma_kx_k$ are equimeasurable for all $a_k$ and $X\stackrel{1}{\subset} L^1[0,1]$, we have
$$
\Big\|\sum_{k=1}^ma_kx_k\Big\|_{X}=\Big\|\sum_{k=1}^ma_kf_k\chi_{E_k}\Big\|_{X_1}.$$
Therefore, from \eqref{in X1} it follows that $x_k$, $k=1,2,\dots,m$, satisfy inequality \eqref{eq-13} for all $a_k\in\mathbb{R}$. Since $\varepsilon>0$ and $m\in\N$ are arbitrary, as a result, the embedding \eqref{eq-12} is proved when $\beta_X<1$. 

It remains to consider the case when $\beta_X=1$ 
and to show \eqref{eq-13} for $p=1$. Note that 
a sketch of the proof of this partial result is given in 
\cite[Proposition 2.b.7]{LT2}. For the reader's convenience, we provide some details.

We will use a simple duality argument. Clearly, we can assume that $X\ne L^1$. Therefore, the K\"{o}the 
dual $X'\ne L_\infty$ and hence the closure $(X')_0$ of $L_\infty$ in $X'$ is 
a separable r.i.\ space. Moreover, since $\beta_X=1$, we have $\alpha_{(X')_0}=0$ (see e.g. \cite[Theorem~II.4.11]{KPS}). 

Let $m\in\N$ and $\varepsilon\in (0,1)$ be arbitrary. 
As was proved above,  there exist $y_k\in (X')_0$, $k=1,2,\dots,m$, 
$\textrm{supp}\,y_i\cap \textrm{supp}\,y_j=\varnothing$ 
for $i\ne j$, such that for all $b_k\in\mathbb{R}$ we have
\begin{equation}\label{eq-20} 
(1+\varepsilon)^{-1}\|(b_k)_{k=1}^m\|_\infty\le \Big\|
\sum_{k=1}^mb_ky_k\Big\|_{(X')_0}\le (1+\varepsilon) \|(b_k)_{k=1}^m\|_\infty.
\end{equation} 
Since $((X')_0)^*=((X')_0)'=X$ (see Section \ref{prel2}), we can find pairwise disjoint and equimeasurable functions $x_k\in X$, $k=1,2,\dots,m$, such that $\|x_k\|_X=1/\|y_k\|_{(X')_0}$ 
and $\int_0^1 x_k(t)y_k(t)\,dt=1$, $k=1,2,\dots,m$ (in particular, this implies that $\int_0^1 x_i(t)y_j(t)\,dt=0$, for $i\ne j$). Moreover, for 
given $a_k\in\mathbb{R}$, $k=1,2,\dots,m$, we take 
$b_k\in\mathbb{R}$, $k=1,2,\dots,m$, so that $\|(b_k)_{k=1}^m\|_\infty=1$ 
and $\|(a_k)_{k=1}^m\|_1 =\sum_{k=1}^m a_kb_k$.
Then, from \eqref{eq-20} it follows that
\begin{eqnarray*}
\Big\|\sum_{k=1}^m a_kx_k\Big\|_X 
&\ge& 
\int_0^1\Big(\sum_{k=1}^m a_kx_k(t)\Big)
\Big(\sum_{i=1}^m b_iy_i(t)\Big)\,dt\cdot \Big\|\sum_{i=1}^m b_iy_i\Big\|_{(X')_0}^{-1}
\\ &\ge& 
(1+\varepsilon)^{-1}{\sum_{k=1}^m a_kb_k}=(1+\varepsilon)^{-1}{\|(a_k)_{k=1}^m\|_1}.
\end{eqnarray*}
In the opposite direction, using \eqref{eq-20} once more, we have
$$
\Big\|\sum_{k=1}^m a_kx_k\Big\|_X \le \sum_{k=1}^m |a_k|\|x_k\|_X
= \sum_{k=1}^m \frac{|a_k|}{\|y_k\|_{(X')_0}}
\le (1+\varepsilon)\|(a_k)_{k=1}^m\|_1.
$$
Thus, we get \eqref{eq-13} also in the case $p=1$, 
and hence the proof of the embedding \eqref{eq-12} is completed.
\end{proof}

In conclusion, we give, as an application of Theorem \ref{t-2}, a description of the set of $p$ such that $\ell^p$ is symmetrically finitely represented in Orlicz and Lorentz spaces (for their definition see Section \ref{prel2}).

Recall that every Orlicz space $L_N=L_N[0,1]$ is of fundamental type (see e.g. \cite{Boyd} or \cite[Theorem~4.2]{Ma-85}) and its fundamental function can be calculated by the formula $\phi_{L_N}(t)=1/N^{-1}(1/t)$, $0<t\le 1$, where $N^{-1}$ is the inverse function for $N$. Moreover, an Orlicz space $L_N$ on $[0,1]$ is separable if and only if the function $N$ satisfies the $\Delta_2^\infty$-condition (see \cite[\S\,II.10]{KR} or Section \ref{prel2}).

\begin{theorem}\label{Theorem 4b}
Let $L_N=L_N[0,1]$ be an Orlicz space such that the function $N$ satisfies the $\Delta_2^\infty$-condition. Then, ${\mathcal F}(L_N)=[1/\beta_{N},1/\alpha_{N}]$, where
\begin{align*}
\alpha_{N} &= -\lim_{n\to\infty}\frac{1}{n}\log_2
\sup_{k\le 0}\frac{N^{-1}(2^{k-n})}{N^{-1}(2^{k})}, 
&\beta_{N} &=\lim_{n\to\infty}\frac{1}{n}\log_2 \sup_{k\le 0}
\frac{N^{-1}(2^{k})}{N^{-1}(2^{k-n})}.
\end{align*}
\end{theorem}

Now, suppose that $1\le q<\infty$ and $\psi$ is an increasing concave function on $[0,1]$ such that $\psi(0)=0$. Then, it is easy to check (see also \cite[\S\,II.4.4]{KPS} or \cite[p.~28]{Ma-85}) that the Lorentz space $\Lambda_q(\psi)=\Lambda_q(\psi)[0,1]$ is a separable r.i.\ space of fundamental type. Since $\phi_{\Lambda_q(\psi)}=\psi^{1/q}$, applying Theorem \ref{t-2}, we get the following result:

\begin{theorem}\label{Theorem 4a}
Let $1\le q<\infty$, and let $\psi$ be an increasing concave function on $[0,1]$ such that $\psi(0)=0$. Then, ${\mathcal F}(\Lambda_q(\psi))=[1/\beta_{\psi,q},1/\alpha_{\psi,q}]$, where
\begin{align*}
\alpha_{\psi,q} &= -\lim_{n\to\infty}\frac{1}{n}\log_2
\sup_{k\le 0}\Big(\frac{\psi(2^{k-n})}{\psi(2^{k})}\Big)^{1/q}, 
&\beta_{\psi,q} &=\lim_{n\to\infty}\frac{1}{n}\log_2 \sup_{k\le 0}
\Big(\frac{\psi(2^{k})}{\psi(2^{k-n})}\Big)^{1/q}.
\end{align*}
\end{theorem}




\end{document}